\documentclass[a4paper,11pt]{amsart}

\usepackage{amsmath,amssymb,amsbsy,amsfonts,amsthm,latexsym,amsopn,amstext,
  amsxtra,euscript,amscd,color,bm}
\usepackage[colorlinks,linkcolor=blue,anchorcolor=blue,citecolor=blue]{hyperref}
\hypersetup{breaklinks=true}

\def\ov#1{{\overline{#1}}}

\newcommand{\h}{{\operatorname{h}}}

\newcommand{\Res}{{\operatorname{Res}}}
\newcommand{\Orb}{{\operatorname{Orb}}}

\newcommand{\bfa}{{\boldsymbol{a}}}

\newcommand{\bfi}{{\boldsymbol{i}}}

\newcommand{\bft}{{\boldsymbol{t}}}

\newcommand{\bfw}{{\boldsymbol{w}}}
\newcommand{\bfx}{{\boldsymbol{x}}}

\newcommand{\bfF}{{\boldsymbol{F}}}
\newcommand{\bfG}{{\boldsymbol{G}}}
\newcommand{\bfR}{{\boldsymbol{R}}}

\newcommand{\bfT}{{\boldsymbol{T}}}
\newcommand{\bfU}{{\boldsymbol{U}}}
\newcommand{\bfX}{{\boldsymbol{X}}}

\newcommand{\bfrho}{{\boldsymbol{\rho}}}

\newcommand{\PreP}{{\mathrm{PrePer}}}

\newcommand{\bfzero}{{\boldsymbol{0}}}

\newfont{\teneufm}{eufm10}
\newfont{\seveneufm}{eufm7}
\newfont{\fiveeufm}{eufm5}
\newfam\eufmfam
                \textfont\eufmfam=\teneufm \scriptfont\eufmfam=\seveneufm
                \scriptscriptfont\eufmfam=\fiveeufm
\newtheorem{theorem}{Theorem}
\newtheorem{lemma}[theorem]{Lemma}

\newtheorem{cor}[theorem]{Corollary}

\newtheorem{rem}[theorem]{Remark}

\numberwithin{table}{section}
\numberwithin{equation}{section}
\numberwithin{figure}{section}
\numberwithin{theorem}{section}

\def\squareforqed{\hbox{\rlap{$\sqcap$}$\sqcup$}}
\def\qed{\ifmmode\squareforqed\else{\unskip\nobreak\hfil
\penalty50\hskip1em\null\nobreak\hfil\squareforqed
\parfillskip=0pt\finalhyphendemerits=0\endgraf}\fi}

\def\fA{{\mathfrak A}}

\def\cP{{\mathcal P}}

\def\cT{{\mathcal T}}

\def\ord{{\mathrm{ord}}}

\definecolor{olive}{rgb}{0.3, 0.4, .1}
\definecolor{dgreen}{rgb}{0.,0.6,0.}

\newcommand{\ignore}[1]{}

\def \C{\mathbb{C}}
\def \F{\mathbb{F}}
\def \K{\mathbb{K}}

\def \Z{\mathbb{Z}}

\def \R{\mathbb{R}}
\def \Q{\mathbb{Q}}
\def \N{\mathbb{N}}
\def \Z{\mathbb{Z}}

\def\mand{\qquad\mbox{and}\qquad}

\def\\{\cr}
\def\({\left(}
\def\){\right)}
\def\fl#1{\left\lfloor#1\right\rfloor}

\begin{document}
\title[Orbits Modulo Primes]
{Orbits of Polynomial Dynamical Systems Modulo Primes}

\author[Chang]
{\sc Mei-Chu Chang}
\address{Chang: Department of Mathematics, University of California.
Riverside,  CA 92521, USA}
\email{mcc@math.ucr.edu}
\urladdr{\url{http://mathdept.ucr.edu/faculty/chang.html}}

\author[D'Andrea]{Carlos D'Andrea}
\address{D'Andrea: Departament de Matem\`atiques i Inform\`atica, Universitat de Bar\-ce\-lona.
Gran Via~585, 08007 Barcelona, Spain}
\email{cdandrea@ub.edu}
\urladdr{\url{http://atlas.mat.ub.es/personals/dandrea}}

\author[Ostafe] {Alina Ostafe} 
\address{Ostafe: School of Mathematics and Statistics, University of New South Wales. 
Sydney, NSW 2052, Australia}
\email{alina.ostafe@unsw.edu.au}
\urladdr{\url{http://web.maths.unsw.edu.au/~alinaostafe}}

\author[Shparlinski]{Igor E. Shparlinski} 
\address{Shparlinski: School of Mathematics and Statistics, University of New South Wales.
Sydney, NSW 2052, Australia}
\email{igor.shparlinski@unsw.edu.au}
\urladdr{\url{http://web.maths.unsw.edu.au/~igorshparlinski}}

\author[Sombra]{Mart{\'\i}n~Sombra}
\address{Sombra: ICREA. Passeig Llu\'is Companys 23, 08010 Barcelona, Spain \vspace*{-2.5mm}}
\address{Departament de Matem\`atiques i
Inform\`atica, Universitat de Barcelona. Gran Via 585, 08007
Bar\-ce\-lo\-na, Spain}
\email{sombra@ub.edu}
\urladdr{\url{http://atlas.mat.ub.es/personals/sombra}}



\date{\today}

\begin{abstract} We present lower bounds for the orbit length of
 reduction modulo primes of parametric polynomial dynamical systems
  defined over the integers, under a suitable hypothesis on its set of
  preperiodic points over $\C$.  Applying recent results of Baker and DeMarco~(2011) 
  and  of Ghioca, Krieger, Nguyen and Ye~(2017),  we obtain explicit
  families of parametric polynomials and initial points such that the
  reductions modulo primes have long orbits, for all but a finite
  number of values of the parameters. This generalizes a previous
  lower bound due to Chang~(2015). As a by-product, we also 
  slighly improve a result of Silverman~(2008) and recover 
  a  result of Akbary and Ghioca~ (2009) 
  as special extreme cases of our estimates. 
\end{abstract}

\maketitle

\vspace{-6mm}
\section{Introduction}

Recently, there has been active interest in the study of orbits of
reductions modulo primes of algebraic dynamical systems defined over
$\Q$, see~\cite{AkGh,BGHKST,Chang,DOSS,Silv2}.  In this paper, we obtain
lower bounds for the orbit length of the reduction modulo primes of
dynamical systems defined by polynomials with integer coefficients,
under a suitable hypothesis on its set of preperiodic points over
$\C$.

One of the first results in this subject is due to 
Silverman~\cite{Silv2}, where he studies the orbit length for the the
reduction modulo a prime $p$ of a dynamical system on a
quasiprojective variety over a number field and a non-preperiodic
point.  In particular, he gives a weak lower bound for the length of
these orbits that is satisfied for every
$p$~\cite[Corollary~12]{Silv2}, and a stronger one that is satisfied
for almost all $p$, in the sense of the analytic
density~\cite[Theorem~1]{Silv2}. This latter lower bound has been slightly
improved by Akbary and Ghioca~\cite{AkGh}, who also show that it
holds for almost all $p$ in the sense of the natural density of
primes.

In~\cite{Chang}, Chang has given a result of a new type involving two
distinct orbits.  Let $F = X^d + T, \ G = X^d + a \in \Z[X,T]$ for a
fixed integer $d \ge 2$ and $a\in \Z[T]\setminus \Z$ with
$a^{d-1}\ne T^{d-1}$.  For a prime $p$, we denote by $\ov \F_p$ the
algebraic closure of $\F_p$. For $t\in \ov \F_{p}$, we set $F_{t}$ for
the map $\ov \F_{p}\to \ov\F_{p}$ defined by
$x\mapsto F_{t}(x)=F(x,t)$, and similarly for $G_{t}$.
By~\cite[Theorem~1]{Chang}, there are constants $c_1,c_2>0$ depending
only on $d$ and $a$ such that, for almost all $p$ (in the sense of the
natural density of primes) there is a set $\cT \subseteq \ov\F_p$ with
$\# \cT \le c_1$ such that, for all $t \in\ov\F_p \setminus \cT$,
\begin{equation}
\label{eq:Chang}
\max\left\{\# \Orb_{F_t}(0), \# \Orb_{G_t}(0)\right \} \ge c_2 \log p,
\end{equation}
where $ \Orb_{F_t}(0)$ and $\Orb_{G_t}(0)$ denote the orbits of
the point $0\in \ov\F_{p}$ in the dynamical systems given by the
iterations of $F_{t}$ and of $G_{t}$, respectively.  This theorem
relies on a previous result of Ghioca, Krieger and Nguyen~\cite{GKN16}
on the finiteness of the set of $t \in \C$ for which
$0 \in \PreP_\C(F_t) \cap \PreP_\C(G_t)$, the intersection of the sets
of preperiodic points of $F_t$ and of $G_{t}$.

Inspired by this result, in the present paper we study the length of
the orbits of the reduction of several parametric dynamical systems
and several starting points. In more precise terms, let
$\bfX=(X_{1},\ldots, X_{m})$ and $\bfT=(T_{1},\ldots, T_{n})$ be groups
of variables and, for $\nu =1, \ldots, r$, let
$\bfF_{\nu}=(F_{\nu,1},\ldots, F_{\nu,m}) \in \Z[\bfX, \bfT]^{m}$,
that we consider as family of $n$-parametric systems of $m$-variate
polynomials. Indeed, given a field $\K$ and a point $\bft=(t_{1},\ldots, t_{n})\in \K^{n}$, we denote
by $\bfF_{\bft}$ the map $\K^{m}\rightarrow \K^{m}$ defined, for
$\bfx\in \K^{m}$, by $\bfF_{\bft}(\bfx)=\bfF(\bfx,\bft)$.
Hence, the system $\bfF$ defines an $n$-parametric family of
polynomial dynamical systems on $\K^{m}$.  

Given a subset $S\subseteq \C^{m}$, 
an important problem in this context is to understand the size and the
structure of the set of points $\bft\in \C^n$ such that
\begin{equation}
\label{eq:9}
S\subseteq \bigcap_{\nu=1}^{r}
\PreP_\C(\bfF_{\nu,\bft}) .
\end{equation} 
Some particular cases of this problem have been  studied by Ghioca, Krieger and
Nguyen~\cite{GKN16}, Ghioca, Krieger, Nguyen and Ye~\cite{GKNY17}, and
Baker and DeMarco~\cite{BDeM11}.  Indeed, the set of preperiodic
points of an algebraic dynamical system over $\C$ is a classical
object of study.  Most of the results and conjectures in this subject
hint that, under suitable hypothesis, this set of preperiodic points
should be rather small, see also~\cite{BDeM13,GHT13,GHT15,GNT15,Ing12}.
The sparsity of these sets suggests that the set of parameters $\bft$ such 
that~\eqref{eq:9} holds should be small, typically finite or empty.

Our first main result in this paper (Theorem~\ref{thm:1}) gives a
lower bound for the orbit length of the reduction modulo primes of
algebraic dynamical systems depending on $n$ parameters, under the
assumption that the set of parameters $\bft\in \C^{n}$
satisfying~\eqref{eq:9} for a given subset of starting integer points
$S\subseteq \Z^{m}$ is finite.  Our proof consists of translating the
condition about the lengths of the orbits into a system of polynomial
equations with integer coefficients, to which we apply a result by
D'Andrea, Ostafe, Shparlinski and Sombra~\cite[Theorem~2.1]{DOSS}.

As a consequence, we recover a result in~\cite{AkGh}, and slightly
improve a result in~\cite{Silv2} (Corollaries~\ref{cor:1} and~\ref{cor:2}).  
Combined with results in~\cite{GKNY17} and in~\cite{BDeM11}, this gives explicit families of parametric polynomials
and initial points such that the reductions modulo primes have long
orbits, for all but a finite number of values for the parameters
(Corollaries~\ref{cor:5} and~\ref{cor:3}). In addition, Corollary~\ref{cor:5}
contains Chang's lower bound~\eqref{eq:Chang} as a particular case, and
sharpens the constant $c_{2}$ therein.

Our second main result (Theorem~\ref{thm:Cycl 2 Poly Univ}) applies to
the case  $n=1$, that is, to systems of polynomials depending on one
parameter. Here, we can strengthen Theorem~\ref{thm:1} to a
result that is valid for every prime.  Its proof follows by applying a
divisibility property for the resultant of two polynomials whose reduction
modulo a prime has several common roots, due to G\'omez-P\'erez,
Guti\'errez, Ibeas and Sevilla~\cite{GGIS}.

\section{Statement of the main results} \label{sec:stat-main-results}

Boldface symbols denote finite sets or sequences of objects, where the
type and number is clear from the context. 
For $m\ge 1$ and $n\ge 0$ we set $\bfX = (X_{1},\ldots, X_{m})$ and
$\bfT=(T_{1},\ldots, T_{n})$, which we consider as groups of {\it
  variables\/} and of {\it parameters\/}, respectively.

Given a system $\bfF=(F_{1}, \ldots, F_{m})\in \Z[\bfX,\bfT]^{m}$,
its iterations are given by
$$
\bfF^{(0)}=\bfX \mand
\bfF^{(k)}= \bfF( \bfF^{(k-1)}, \bfT) \quad \text{ for } k\ge 1. 
$$
For a field $\K$ and a point $\bft=(t_{1},\ldots, t_{n})\in \K^{n}$, we
consider the map
\begin{equation}
  \label{eq:12}
  \bfF_{\bft}\colon  \K^{m}\longrightarrow \K^{m}, \quad \bfx\longmapsto
  \bfF(\bfx,\bft).
\end{equation}

Hence, $\bfF$ defines an $n$-parametric family of polynomial dynamical
systems on $\K^{m}$. Given a vector $\bfw\in \K^{m}$, we denote by
$\Orb_{\bfF_{\bft}}(\bfw)$ the \emph{orbit} of $\bfw$ under the map
in~\eqref{eq:12}. Such a point is \emph{preperiodic} with respect to
$\bfF_{\bft}$ if its orbit is finite, and the set of these preperiodic
points is denoted by $\PreP_\K(\bfF_{\bft})$.  We refer
to~\cite{AnKhr,Schm,Silv1} for a background on these dynamical
systems.

As usual, we use  $\ord_{p}z $ to denote the $p$-adic order of $z \in \Z$

Although we are mostly interested in the case of $n$ parameters with
$n \ge 1$, we sometimes consider the non-parametric case when $n=0$
(thus $\K^0 = \{0\}$) and recover a result in~\cite{AkGh} and slightly
improve another result in~\cite{Silv2}.

For a vector $\bfa \in \Z^\ell$, we define its \emph{height}, denoted
by $\h(\bfa)$, as the logarithm of the maximum of the absolute values
of its coordinates, if $\bfa\ne\bfzero$, and as $0$ otherwise.  For a
polynomial $G$ with integer coefficients, its \emph{height}, denoted
by $\h(G)$, is defined as the height of its vector of coefficients.
For a family of polynomials $\bfG=(G_{1},\ldots, G_\ell)$ with
integer coefficients, we respectively define its \emph{degree} and
\emph{height} as
$$
  \deg \bfG=\max_{1\le i\le \ell}\deg G_{i} \mand \h(\bfG)=\max_{1\le
    i\le \ell}\h(G_{i}).
$$

Given functions 
$$
f,g\colon \N\longrightarrow \R, 
$$
the symbol $f\ll g$ means that there is a constant $c\ge0$ such that
$|f(k)|\le c \, g(k)$ for all $k\in \N$.  To emphasize the dependence of
the implied constant $c$ on a list of parameters $\bfrho$, we write
$f\ll_{\bfrho} g$.

We first present a lower bound for the length of the orbits of
reduction modulo primes of several parametric multivariate polynomial
systems and several initial points.

\begin{theorem} \label{thm:1} Let $\bfF_{\nu} \in \Z[\bfX,\bfT]^{m}$,
$\nu =1, \ldots, r$, be a family of $r\ge 1$ parametric
systems of polynomials and $\bfa_{j}\in \Z^{m}$, $j=1,\ldots, s$, a
family of $s\ge 1$ integer vectors, such that the set
\begin{equation}
  \label{eq:13}
\{\bft\in \C^n~:~\bfa_{j}\in \PreP_\C(\bfF_{\nu,\bft})
\text{ for all } \nu, j\}
\end{equation}
is empty if $n=0$, or finite if $n\ge1$. Set $\kappa$ for the
cardinality of this set, and let $d\ge
\max\{2,\deg \bfF_{\nu}\}$ for all $\nu$ and $h\ge \max\{\h(\bfF_{\nu}), \h(\bfa_{j})\}$
for all $\nu$ and $j$. Let also $L\ge 1$. Then there is an integer
$\fA_{L}\ge 1$ with
$$
\log \fA_{L} \ll_{m,n,r,s,h}
\begin{cases} Ld^{L} & \mathrm{ if } \ n= 0, \\ L^{3n+3}d^{(3n+2)L} &
\mathrm{ if } \ n\ge 1,
\end{cases}
$$
 such that,
for every prime $p$ not dividing $ \fA_{L}$, for all but at most
$\kappa$ values of $\bft \in \ov \F_p^n$,
$$
\max_{\substack{1\le \nu \le r \\1\le j \le s}}
\#\Orb_{\bfF_{\nu,\bft}}(\bfa_{j} \bmod p) > L.
$$
\end{theorem}


\begin{rem}
  \label{rem:1}
  When $n=0$, the case $r=s=1$ already contains the cases when $r$ and
  $s$ are arbitrary. Indeed, we recall that $\C^{0}=\{0\}$
  and so $\bft=0$. Now, let $\bfF_{\nu} \in \Z[\bfX]^{m}$,
  $\nu =1, \ldots, r$, and $\bfa_{j}\in \Z^{m}$, $j=1,\ldots, s$, such
  that the set in~\eqref{eq:13} is empty.  Note that accordingly to our general convention
  there is only one possible specialisation   of $\bfF_{\nu}$ with $\bft=0$ 
  and  $\bfF_{\nu, 0} = \bfF_{\nu}$.  The previous condition then implies that there exist
  $\nu_{0}$ and $j_{0}$ such that
$$
    \bfa_{j_{0}}\notin  \PreP_\C(\bfF_{\nu_{0},0}). 
$$
Theorem~\ref{thm:1} applied to this system and this initial point
implies that, for all $p\nmid \fA_L$, 
$$
\#\Orb_{\bfF_{\nu,0}}(\bfa_{j} \bmod p) > L,
$$
which gives the conclusion for the whole families $\bfF_{\nu}$,
$\nu=1,\ldots, r$, and $\bfa_{j}$, $j=1,\ldots, s$. 
\end{rem}

We have the following result for all primes.

\begin{cor} \label{cor:1} With conditions as in Theorem~\ref{thm:1},
  for any $0<\varepsilon <\frac{1}{(3n+2)\log d}$, 
  there exists a constant $c$ depending only on $m,n,r,s, h$ and $\varepsilon$ such
  that, for all $p\ge c$ and all but at most $\kappa$ values of
  $\bft\in \ov \F_p^n$,
 $$
\max_{\substack{1\le \nu \le r \\1\le j \le s}}\# \Orb_{\bfF_{\nu,
    \bft}}( \bfa_{j} \bmod p) > \varepsilon\, {\log\log p}.
$$
When $n=0$, this conclusion also holds for any
$0<\varepsilon <\frac{1}{\log d}$.
\end{cor}

This result applied to a polynomial system $\bfF\in \Z[\bfX]^{m}$ and
a point $\bfa\in \Z^{m}$ with infinite orbit with respect to the map
$\bfF\colon \C^{m}\to \C^{m}$, shows that there is a constant $c(m,h)$ such
that, for every $p \ge c(m,h)$,
 $$
\#\Orb_{\bfF}( \bfa_{j} \bmod p) > \frac{\log\log p}{\log d}.
$$
This refines the lower bound in~\cite[Corollary~12]{Silv2} for a
dynamical system on the affine space defined by polynomials with
integer coefficients, by giving its explicit dependence on the degree
of $\bfF$.
 
For a subset $\cP$ of the set of primes, its
\emph{natural density} is defined as the real number
$$
  \lim_{Q\to \infty} \frac{\# \{ p\in \cP~:~p\le Q\}}{\# \{ p\
    \mathrm{ prime} ~:~ p\le Q\}},
$$
whenever this limit exists.  We can also deduce from
Theorem~\ref{thm:1} the following stronger lower bound for the length
of the orbits of the system $\bfF$ modulo a prime $p$ that is valid
for almost all  primes $p$, in the sense of the natural density of this set.

\begin{cor} \label{cor:2} 
 Under the conditions of Theorem~\ref{thm:1},
  for any $0<\varepsilon <\frac{1}{(3n+2)\log d}$, the
  set of primes $p$ such that, for all but at most $\kappa$ values of
  $\bft\in \ov \F_p^n$,
 $$
\max_{\substack{1\le \nu \le r \\1\le j \le s}} 
\# \Orb_{\bfF_{\nu, \bft}}( \bfa_{j} \bmod p) \ge \varepsilon \log p,
$$
has natural density 1. When $n=0$, this conclusion also holds for any
$0<\varepsilon <\frac{1}{\log d}$.
 \end{cor}

 For a  polynomial system $\bfF\in \Z[\bfX]^{m}$ and a
 point $\bfa\in \Z^{m}$ with infinite orbit over $\C$, Corollary~\ref{cor:2} 
  recovers~\cite[Theorem~1.1(1)]{AkGh}. 

 The result of Ghioca, Krieger, Nguyen and Ye in~\cite{GKNY17}
 mentioned in the introduction implies that, for $d\ge 2$ and
 $u,v\in\Z[T]\setminus \Z$ such that $u^{d-1}\ne v^{d-1}$, the set of
 $t\in \C$ such that the point $0\in \C$ is preperiodic both for the
 map $x\mapsto x^{d}+u(t)$ and the map $x\mapsto x^{d}+v(t)$, is
 finite.

The following result is a direct consequence of
Corollaries~\ref{cor:1} and~\ref{cor:2}.  It generalizes Chang's lower
bound~\eqref{eq:Chang} to a larger family of pairs of polynomials and,
moreover, it refines the value of the constant $c_{2}$ in that lower
bound.

 \begin{cor}
   \label{cor:5}
   Let $d\ge 2$ and $u,v\in\Z[T]\setminus \Z$ such that
   $u^{d-1}\ne v^{d-1}$. Then, for any $0<\varepsilon <\frac{1}{ 5\log d}$, there exists $\kappa \ge 0$ such that,
   for every sufficiently large $p$ and all but at most $\kappa$
   values of $t\in \ov \F_p$,
 $$
 \max\left\{\# \Orb_{x^{d}+u(t)}(0), \# \Orb_{x^{d}+v(t)}(0)\right \} >
 \varepsilon \, {\log\log p}.$$
 Furthermore, the set of primes $p$ such that, for all but at most
 $\kappa$ values of $t\in \ov \F_p$,
 $$
     \max\left\{\# \Orb_{x^{d}+u(t)}(0), \# \Orb_{x^{d}+v(t)}(0)\right \}  \ge \varepsilon \log p,
$$
has natural density 1.
 \end{cor}
 
 Another instance where our results can be applied is given by the
result of Baker and DeMarco~\cite[Theorem~1.1]{BDeM11} 
mentioned in the introduction: given
 $d\ge 2$ and $a_{1},a_{2}\in \Z$, the set of $t\in\C$ such that both
 $a_1$ and $a_2$ are preperiodic for the map $X\mapsto X^d+t$ is
 infinite if and only if $a_1^d=a_2^d$. The
 following result is also a direct consequence of
 Corollaries~\ref{cor:1} and~\ref{cor:2}.

\begin{cor}
  \label{cor:3}
  Let $d\ge 2$ and $a_{1},a_{2}\in \Z$ with $a_{1}^{d}\ne
  a_{2}^{d}$.
  Then, for any $0<\varepsilon <\frac{1}{ 5\log d}$,  there exists $\kappa \ge 0$ such that, for every
  sufficiently large $p$ and all but at most $\kappa$ values of
  $t\in \ov \F_p$,
 $$
\max \left\{ \# \Orb_{X^{d}+t}(a_{1} \bmod p) , \# \Orb_{X^{d}+t}(a_{2}
\bmod p) \right\} >  \varepsilon \, {\log\log p}.
$$
Furthermore, the
set of primes $p$ such that, for all but at most $\kappa$ values of
$t\in \ov \F_p$,
 $$
\max \{ \# \Orb_{x^{d}+t}(a_{1} \bmod p) , \# \Orb_{x^{d}+t}(a_{2} \bmod p) \} \ge \varepsilon \log p,
$$
has natural density 1.
\end{cor}

For systems depending on a single parameter $T$, we can strengthen
Theorem~\ref{thm:1} to a result that is valid for every prime. 

\begin{theorem}
\label{thm:Cycl 2 Poly Univ} 
Let $\bfF_{\nu} \in \Z[\bfX,T]^{m}$,
$\nu =1, \ldots, r$, be a family of $r\ge 1$ parametric
systems of polynomials and $\bfa_{j}\in \Z^{m}$, $j=1,\ldots, s$, a
family of $s\ge 1$ integer vectors, such that the set
\begin{equation}
\label{eq:6}
\{t\in \C ~:~\bfa_{j}\in \PreP_\C(\bfF_{\nu,t})
\text{ for all } \nu, j\}
\end{equation} 
is finite. Set $\kappa$ for the cardinality of this set, and
let $d\ge \max\{2,\deg \bfF_{\nu}\}$ for all $\nu$ and
$h\ge \max\{\h(\bfF_{\nu}), \h(\bfa_{j})\}$ for all $\nu$ and $j$.  Let also
$L\ge 1$.  Then there is an integer $\fA_{L}\ge 1$ with
$$
\log \fA_{L} \ll_{m,r,s,h} L^{2}d^{2L},
$$
such that, for every prime $p$, for all but at
most $\kappa+ \ord_{p}\,\fA_{L}$ values of $t \in \ov \F_p$,
$$
\max_{\substack{1\le \nu \le r \\ 1\le j \le s}}\#
\Orb_{\bfF_{\nu,t}}(\bfa_{j} \bmod p) > L.
$$
\end{theorem}

Theorem~\ref{thm:Cycl 2 Poly Univ}  contains Theorem~\ref{thm:1} for systems depending on a
single parameter, with a better control for the integer $\fA_{L}$:
this latter result corresponds to the primes $p$ such that
$\ord_{p}\,\fA_{L} =0$. 

As a consequence of Theorem~\ref{thm:Cycl 2 Poly Univ}, we obtain the following result
valid for all primes, and which  is a sharper version of
Corollary~\ref{cor:2} for the case of $n=1$ parameter. 

\begin{cor}
  \label{cor:4} 
  With conditions as in Theorem~\ref{thm:Cycl 2 Poly Univ}, for any
  $0<\varepsilon< \frac{1}{2\log d}$, there exists $0<\gamma<1$ such
  that, for every $Q\ge 2$ and every prime $p\le Q$, the number of
  values  of $t \in \ov \F_p$ such that 
$$
\max_{\substack{1\le \nu \le r \\ 1\le j \le s}}
\#\Orb_{\bfF_{\nu,t}}(\bfa_{j} \bmod p) \le  \varepsilon \log p
$$
is bounded by    $\kappa+ c_{p}$, with 
$$
\sum_{p\le Q} c_{p}\ll_{m,n,r,s,h} Q^{\gamma}.
$$ 
\end{cor}

\section{Preliminaries}\label{sec:preliminaries}

In this section, we gather some bounds on the heights and degrees of
some polynomials. We also need some rather general statements about
the reduction modulo primes of systems of multivariate polynomials
and about the divisibility of resultants.

We start with  bounds for the height of sums and products of
polynomials, whose proof can be derived from~\cite[Lemma~1.2]{KPS}. 

\begin{lemma}
\label{lem:HeightProd}
Let $G_{i}\in \Z[T_{1},\ldots, T_{n}]$, $i=1,\ldots, s$. Then 
\begin{enumerate}
\item  \label{item:1} $\displaystyle{\h \(\sum_{i=1}^s G_{i}\) \le \max_{1\le
    i\le s} \h(G_i) +\log s;}
   $
    
\item\label{item:2} $  \displaystyle{- 2  \log(n+1) \sum_{i=1}^s\deg G_{i} \le \h \(\prod_{i=1}^s G_{i}\)
 -\sum_{i=1}^s \h\( G_{i}\)
 \le  \log(n+1) \sum_{i=1}^s\deg G_{i}  . }$
\end{enumerate}
\end{lemma}

We also need the upper bound from~\cite[Lemma~3.4]{DOSS} for the degree
and the height of iterations of polynomial dynamical systems.

\begin{lemma}
\label{lem:HeighIter-Poly}
Let $G_{i}\in \Z[T_{1},\ldots, T_{n}]$, $i=1,\ldots, n$, be polynomials
of degree at most $d\geq2$ and height at most $h$.  Set
$\bfG=(G_{1},\ldots, G_{n})$ and, for $k\ge0$, let $\bfG^{(k)}$
denote the $k$-th iterate of $\bfG$. Then
$$
\deg \bfG^{(k)} \le  d^k \mand \h \(\bfG^{(k)}\) \le h \frac{d^k-1}{d-1}+d(d+1)\frac{d^{k-1}-1}{d-1}\log(n+1). 
$$
\end{lemma}

Crucial to our strategy is the following result on the reduction
modulo primes of systems of multivariate polynomials over the
integers, whose proof relies on the arithmetic Nullstenllensatz
from~\cite{DKS}.

\begin{theorem}[{\cite[Theorem~2.1]{DOSS}}] 
\label{thm:2}
  Let $G_{i}\in \Z[T_{1},\ldots, T_{n}]$, $i=1,\ldots, s$, be $n\ge 1$ polynomials
  of degree at most $d\geq2$ and height at most $h$, whose zero set
  in $\C^{n}$ has a finite number $\kappa$ of distinct points.  Then
  there is an integer $\fA\ge 1$ with
$$
\log \fA \le (11n + 4 ) d^{3n +1} h +(55 n+99) \log ((2n+5)s) d^{3n
  +2}
$$
such that, if $p$ is a prime not dividing $\fA$, then the zero set in
$\ov \F_{p}^{n}$ of the polynomials $G_i\bmod{p}$, $i=1,\ldots, s$,
consists of exactly $\kappa$ distinct points.
\end{theorem}

Given two univariate polynomials $F_{1}, F_{2} \in \Z[T]$, if
their reductions $F_{i} \bmod{p}$, $i=1,2$, have a common zero in
$\ov\F_{p}$, then their resultant $\Res(F_{1},F_{2})$ is divisible by
$p$.  The following result refines this property for polynomials
whose reduction modulo $p$ has several common roots.
%
\begin{theorem}[{\cite{GGIS}}] 
 \label{thm:3}
  Let $A$ be a unique factorization domain with field of fractions
  $K$, $p\in A$ an irreducible element, and $F_{1}, F_{2} \in A[T]$
  two univariate polynomials whose reductions modulo $p$ do not vanish
  identically and have at least $N$ common roots in $\ov K$, counted
  with multiplicities.  Then $p^N \mid \Res(F_{1},F_{2})$.
\end{theorem}

Indeed, for our application it is sufficient to use the result
of~\cite[Lemma~5.3]{KS99} taking only into account the number of
different roots of the reductions of the polynomials $F_{i}$ modulo $p$.

\section{Proofs of the main results}

In this section, we prove the results stated in
\S\ref{sec:stat-main-results}. We start with Theorem~\ref{thm:1} and
its consequences.
 
\begin{proof}[Proof of Theorem~\ref{thm:1}]
  Fix $1\le \nu\le r$ and $1\le j\le s$. Given $0\le k\le L-1$, a
  point $\bft\in \C^{n}$ verifies that
  $$
    \bfF_{\nu}^{(L)}(\bfa_{j}, \bft)=\bfF_{\nu}^{(k)}(\bfa_{j},\bft)
  $$
  if and only if it lies in the zero set of the ideal
$$
I_{\nu,j,k}=\(F_{\nu,i}^{(L)}(\bfa_{j}, \bfT)-F_{\nu,i}^{(k)}(\bfa_{j},
\bfT)~:~1\le i\le m \) \subseteq \Z[\bfT].
$$
Hence, 
\begin{math}
\#  \Orb_{\bfF_{\nu},\bft}(\bfa_{j}) \le L
\end{math}
if and only if $\bft$ lies in the zero set of the ideal
$\prod_{k=0}^{L-1} I_{\nu,j,k}$.

For each $\nu=1,\ldots, r$, $\bfi\in \{1,\ldots, m\}^{L}$  and $j=1,\ldots,
s$, consider the polynomial
$$
\Psi_{\nu,\bfi,j}=\prod_{k=0}^{L-1}\(F_{\nu,i_{k+1}}^{(L)}(\bfa_{j}, \bfT) -
F_{\nu,i_{k+1}}^{(k)}(\bfa_{j}, \bfT)\) \in \Z[\bfT].
$$
This gives a set of $r s m^{L}$ generators of the ideal
$$
\sum_{\nu=1}^r \sum_{j=1}^s \prod_{k=0}^{L-1} I_{\nu,j,k} \subseteq \Z[\bfT].
$$  Hence,
for a point $\bft\in \C^{n}$, 
\begin{equation}
\label{eq:4}
\max_{\substack{1\le \nu \le r \\ 1\le j \le s}}\# \Orb_{\bfF_{\nu},\bft}(\bfa_{j}) \le L
\end{equation}
if and only if $\Psi_{\nu,\bfi,j}(\bft)=0$ for all $\nu$, $\bfi$ and
$j$.  Moreover, the set of such parameters  $\bft$ is contained in the set
of $\bft\in \C^{n}$ such that $\bfa_{j}\in \PreP_\C(\bfF_{\nu,\bft})$
for all $\nu$ and $j$.  By hypothesis, this latter set is empty if
$n=0$, and finite if $n\ge 1$.  Hence, the number of possible values of $\bft$'s
satisfying~\eqref{eq:4} is finite and bounded above by the constant
$\kappa$.

For $\nu=1,\ldots, r$, consider the family of $m+n$ polynomials in
$m+n$ variables given by 
$$
\bfR_\nu  =(\bfF_\nu , \bfT) \in \Z[\bfX,\bfT]^{m+n}. 
$$
For $k\ge 0$, we have that
$ \bfR_{\nu}^{(k)}=(\bfF_{\nu}^{(k)}, \bfT)$. Hence, the $k$-th
iteration of the system $\bfF_{\nu}$ with respect to the variables
$\bfX$ can be recovered from the first $m$ coordinate polynomials of
the $k$-th iteration of the system $\bfR_{\nu}$.  Applying
Lemma~\ref{lem:HeighIter-Poly} to $\bfR_\nu $, we deduce that
$\deg \bfF_{\nu}^{(k)}\le d^k$ and
$$
\h(\bfF_{\nu}^{(k)})\le h
\frac{d^k-1}{d-1}+d(d+1)\frac{d^{k-1}-1}{d-1}\log(m+n+1)
\ll_{m,n,r,s,h} d^{k}.
$$
By Lemma~\ref{lem:HeightProd}, for all $\nu,\bfi,j$,
\begin{equation}
  \label{eq:1}
  \deg \Psi_{\nu,\bfi,j} \le Ld^{L} \mand   \h(\Psi_{\nu,\bfi,j}) \ll_{m,n,r,s,h} Ld^{L}.
\end{equation}

When $n=0$, the polynomials $\Psi_{\nu,\bfi,j}$ are constant. As in Remark~\ref{rem:1}, our hypothesis that
there is no $\bft\in \C^{0} = \{0\}$ satisfying~\eqref{eq:4} implies  that there exist
  $\nu_{0}$ and $j_{0}$ such that
$   \bfa_{j_{0}}\notin  \PreP_\C(\bfF_{\nu_{0},0})$, and thus $\Psi_{\nu_0,\bfi,j_0}\ne 0$ for all $\bfi$.
In this case we take $\fA_{L}=\gcd\{\Psi_{\nu_0,\bfi,j_0}~:~\bfi\in \{1,\ldots, m\}^{L}\}$.

When $n\ge 1$, we set $\fA_{L}$ for the positive 
integer given by Theorem~\ref{thm:2} applied to this family of
polynomials, which satisfies
\begin{align*}
\log \fA_{L}  & \le (11n  + 4 )(Ld^{L})^{3n +1} Ld^{L} + (55 n+99) \log \((2n+5) (r s
              m^{L})\)(Ld^{L})^{3n +2} \\
& \ll_{m,n,r,s,h}  L^{3n+3}d^{(3n+2)L}.
\end{align*}
In both cases,  for every prime $p\nmid \fA_{L}$, the system of equations
$$
  \Psi_{\nu,\bfi,j} (\bfa_{j} \bmod p, \bft) = 0
$$
has at most $\kappa$ solutions $\bft\in \ov \F_p^n$.
Similarly as before, 
this is equivalent to the statement that 
$$
\max_{\substack{1\le \nu \le r \\ 1\le j \le s}}\# \Orb_{\bfF_{\nu,\bft}}(\bfa_{j} \bmod p) > L, 
$$
for all but at most $\kappa$ values of  $\bft\in \ov \F_p^n$, which
proves the theorem. 
\end{proof}

\begin{proof}[Proof of Corollary~\ref{cor:1}]
Theorem~\ref{thm:1}
  applied with
$ 
 L=\fl{\varepsilon \, \log\log p}
$ 
implies there is a positive integer $\fA_{L}$ with 
\begin{equation}
  \label{eq:7}
  \log\fA_{L} \ll_{m,n,r,s,h} 
  \begin{cases}
( \varepsilon \, \log \log p)
  (\log p)^{\varepsilon \, \log d} & \text{ if } n=0,\\
 (\varepsilon \, \log \log p)^{3n+3}
  (\log p)^{\varepsilon \, (3n+2) \log d} & \text{ if } n\ge1,    
  \end{cases}
\end{equation}
such that, for all $p\nmid \fA_{L}$, for all but at
most $\kappa$ values of $\bft \in \ov \F_p^n$,
$$
\max_{\substack{1\le \nu \le r \\ 1\le j \le s}}\# \Orb_{\bfF_{\nu, \bft}}( \bfa_{j} \bmod p) >
\varepsilon \, {\log\log p}.
$$

The bound~\eqref{eq:7} implies that there is a constant
$c$, depending on the parameters $m$, $n$, $r$, $s$ and $h$,  such that
$\fA_{L} <p$ for all $p\ge c$.  
For those primes $p$, we have that $p\nmid \fA_{L}$ and  the result follows.
 \end{proof}

\begin{proof}[Proof of Corollary~\ref{cor:2}]
  Let $Q\ge 2$. Theorem~\ref{thm:1} applied with
  $L=\lfloor{\varepsilon \log Q}\rfloor$ implies that there is an
  integer $\fA_{L}\ge1 $ with
\begin{equation}
\label{eq:2}
  \log\fA_{L} \ll_{m,n,r,s,h} 
\begin{cases}
 ( \log Q)  \, Q^{ \varepsilon \log d} & \mathrm{ if } \ n= 0, \\  
 ( \log Q)^{3n+3}  \, Q^{\varepsilon (3n+2) \log d}
& \mathrm{ if } \ n\ge 1,
\end{cases}
\end{equation}
such that, for all $p\le Q$ with  $p\nmid \fA_{L}$, for all but at most $\kappa$
values of $\bft \in \ov \F_p^n$,
\begin{equation}
  \label{eq:8}
 \max_{\substack{1\le \nu \le r \\ 1\le j \le s}}\# \Orb_{\bfF_{\nu, \bft}}( \bfa_{j} \bmod p) >
 \varepsilon \log Q \ge \varepsilon \log p.
\end{equation}
The divisibility $p\mid \fA_{L}$ is possible for at most
$\log\fA_{L} / \log 2$ primes $p$. Hence, the
 bound~\eqref{eq:2}
implies that the set of primes $p\le Q$ not satisfying~\eqref{eq:8} is
of size ${\mathcal O} _{m,n,r,s,h} \(Q^{\gamma}\)$ for an exponent
$0<\gamma<1$. Hence, this subset of primes has natural density 0, and
thus its complement has natural density 1, as stated.
 \end{proof}

We now treat polynomial systems depending on a single parameter $T$.

 \begin{proof}[Proof of  Theorem~\ref{thm:Cycl 2 Poly Univ}]
For each $\nu=1,\ldots, r$, $\bfi\in \{1,\ldots, m\}^{L}$  and $j=1,\ldots,
s$, consider the polynomial
$$
\Psi_{\nu,\bfi,j}=\prod_{k=0}^{L-1}\(F_{\nu,i_{k+1}}^{(L)}(\bfa_{j}, T) -
F_{\nu,i_{k+1}}^{(k)}(\bfa_{j}, T)\) \in \Z[T].
$$
As in the proof of Theorem~\ref{thm:1}, a point $t\in \C$ verifies that
$$
\max_{\substack{1\le \nu \le r \\ 1\le j \le s}}\# \Orb_{\bfF_{\nu},t}(\bfa_{j}) \le L
$$
if and only if $\Psi_{\nu,\bfi,j}(t)=0$ for all $\nu$, $\bfi$ and $j$.
The set of such $t$ is contained in the set~\eqref{eq:6} and,
by the hypothesis on this latter, the number of such values of $t$ is
finite and bounded above by the constant $\kappa$.

As in~\eqref{eq:1}, the number of such polynomials
is $r s m^{L}$, and their degree and height are bounded by
\begin{equation}\label{eq:10}
  \deg \Psi_{\nu,\bfi,j} \ll_{m,r,s,h} L d^{L } \mand \h(\Psi_{\nu,\bfi,j}) \ll_{m,r,s,h} L d^{L }.
\end{equation}

Let $H\in \Z[T]$ be a primitive polynomial that is a greatest common
divisor in $\Q[T]$ of the polynomials $\Psi_{\nu,\bfi,j}$, and write
$$
\Phi_{0}, \ldots, \Phi_{u}
$$
for the distinct nonzero polynomials of the form
$\Phi_{l}=\Psi_{\nu,\bfi,j}/H$ for some $\nu,\bfi$ and $j$.  We have
that $u< rsm^{L}$, and we deduce from~\eqref{eq:10} and
Lemma~\ref{lem:HeightProd}\eqref{item:2} that, for $l=0,\ldots, u$,
$$
  \deg \Phi_{l} \ll_{m,r,s,h} L d^{L } \mand  \h(\Phi_{l}) \ll_{m,r,s,h} L d^{L }.
$$
Let $\bfU=(U_{1},\ldots, U_{u})$ be a group of variables and set
$$\Phi=\sum_{l=1}^{u}U_{l}\Phi_{l} \mand    R=\Res (\Phi_{0}, \Phi) \in \Z[\bfU]. 
$$
Since the polynomials $\Phi_{l}$ are coprime, it follows that $\Phi_{0}$ and $\Phi$
are coprime. Moreover, $\Phi_{0}$ is nonzero and so $R$ is nonzero
too. Using Sylvester's determinantal formula for the resultant and
Lemma~\ref{lem:HeightProd}\eqref{item:2}, we deduce that
$$
\deg R \ll_{m,r,s,h} L^{2}d^{2L} \mand  \h(R) \ll_{m,r,s,h} L^{2}d^{2L},
$$
and we set  $\fA_{L}\in \Z\setminus \{0\}$ as any nonzero coefficient of this
polynomial. 

Let $p$ be a prime and denote by $\Lambda_{p}\subseteq  \ov{\F}_{p} $
the subset of   $t\in  \ov{\F}_{p} $ such that 
$$
\max_{\substack{1\le \nu \le r \\ 1\le j \le s}}\# \Orb_{\bfF_{\nu,\bft}}(\bfa_{j} \bmod p) \le L. 
$$
As before, this coincides with the zero set of the reductions of the
polynomials $\Psi_{\nu,\bfi,j}$ modulo $p$. Let $\kappa+e_p$ be the
cardinality of this set, with $e_{p}\in \Z$. We denote by
$H_{p}, \Phi_{p}, \Phi_{l,p} \in \F_{p}[T]$ the reductions modulo $p$
of $H, \Phi, \Phi_{l}\in \Z[T]$, $l=0,\ldots, u$, respectively.

If $t\in \Lambda_{p}$, then either $H_{p}(t)=0$ or $\Phi_{l}(t)=0$,
$l=0,\ldots, u$. The number of zeros of $H_{p}$ is bounded by
$\deg H\le \kappa$.  The number of common zeros in $\ov{\F}_{p} $ of the
polynomials $\Phi_{l,p}$ coincides with the number of common zeros in
$\ov{\F}_{p} $ of $\Phi_{0,p}$ and $\Phi_{p}$. By Theorem~\ref{thm:3},
this number is bounded above by $\ord_{p}\,R$, the largest power of
$p$ dividing all coefficients of $R$.   In turn, this is also bounded
above by $\ord_{p}\,\fA_{L}$ (which is the $p$-adic order of one of the 
non-zero coefficients of $R$).  It follows that
$$
  \max\{0,e_{p}\}\le \ord_{p}\,\fA_{L},
$$
proving the result. 
 \end{proof}

 \begin{proof}[Proof of Corollary~\ref{cor:4}]
Let $Q\ge 2$. Theorem~\ref{thm:Cycl 2 Poly Univ} applied with $L=\lfloor
\varepsilon \log Q \rfloor$ implies that there is an integer
$\fA_{L}\ge 1$ such that
$$
  \log\fA_{L}\ll_{m,r,s,h} (\log Q)^{2} Q^{ 2 \varepsilon\log d}.
$$
such that, for every  $p$, for all but at
most $\kappa+ \ord_{p}\, \fA_{L}$ values of $t \in \ov \F_p$,
$$
\max_{\nu, j}\#
\Orb_{\bfF_{\nu,t}}(\bfa_{j} \bmod p) > \varepsilon \log
Q \ge  \varepsilon \log p.
$$
The statement follows by taking any $ 2 \varepsilon\log d <\gamma <1$
and $c_{p}=\ord_{p}\,\fA_{L}$.
 \end{proof}

\section*{Acknowledgements}

During the preparation of this paper, Chang was partially supported by the NSF Grants DMS~1301608, 
D'Andrea 
by the Spanish MEC research project MTM2013-40775-P, Ostafe  by the 
UNSW Vice Chancellor's Fellowship, Shparlinski by the
ARC Grant~DP140100118, and Sombra by the Spanish
MINECO research project MTM2015-65361-P.

\end{document}